\documentclass{amsart}
\usepackage{amssymb}

\usepackage[utf8]{inputenc}
\usepackage[T1]{fontenc}

\usepackage{tikz}
\usetikzlibrary{trees}

\usepackage{color}
\usepackage{url}
\usepackage{enumerate}

\theoremstyle{plain} 
\newtheorem{theorem}{Theorem}
 
\newtheorem{corollary}[theorem]{Corollary} 
\newtheorem{proposition}[theorem]{Proposition}

\newtheorem{problem}[theorem]{Problem}
\newtheorem{example}[theorem]{Example}
\theoremstyle{remark}

\theoremstyle{definition}
\newtheorem{definition}[theorem]{Definition}

\newcommand{\F}{\mathcal{F}}
\newcommand{\bR}{\mathbb{R}}
\newcommand{\bZ}{\mathbb{Z}}
\newcommand{\bN}{\mathbb{N}}
\newcommand{\bC}{\mathbb{C}}
\newcommand{\bQ}{\mathbb{Q}}
\newcommand{\bK}{\mathbb{K}}

\newcommand{\cG}{\mathcal{G}}

\newcommand{\rng}{\mathrm{rng}}

\newcommand{\QC}{\mathrm{QC}}

\newcommand{\osc}{\mathrm{osc}}
\newcommand{\card}{\mathrm{card}}
\newcommand{\co}{\mathfrak{c}}
\newcommand{\Cliq}{\mathrm{Cliq}}
\newcommand{\PWD}{\mathrm{PWD}}
\newcommand{\Cont}{\mathrm{C}}
\newcommand{\SQC}{\mathrm{QC}^\ast}

\newcommand{\cl}{\mathrm{cl}}
\newcommand{\restr}{{\restriction}}
\renewcommand{\int}{\mathrm{int}}

\title{Borsik's properties of topological spaces and their applications}

\author{Tomasz Natkaniec}
\address{Institute of Mathematics, Faculty of Mathematics, Physics and Informatics,
University of Gda\'{n}sk, 80-308 Gda\'{n}sk, Poland}
\email{tomasz.natkaniec@ug.edu.pl}

\begin{document}

{
	\renewcommand{\thefootnote}{}
	\footnotetext{Corresponding author's email: tomasz.natkaniec@ug.edu.pl}
}



\begin{abstract}
Let $X$ be an uncountable Polish space. \u{L}ubica Hol\'{a} showed recently that there are $2^\co$ quasi-continuous real valued functions defined on the uncountable Polish space $X$ that are not Borel measurable. Inspired by Hol\'{a}'s result, we are extending it in two directions. First, we prove that if $X$ is an uncountable Polish space and $Y$ is any Hausdorff space with $|Y|\ge 2$ then the family of all non-Borel measurable quasi-continuous functions has cardinality $\ge 2^{\co}$. Secondly, we show that the family of quasi-continuous non Borel functions from $X$ to $Y$ may contain big algebraic structures.
\end{abstract}

\date{24.11.2023}
\maketitle

\section{Introduction}
\subsection{Super quasi-continuous functions}
Let $X$ and $Y$ be topological spaces. We  will  consider the following families of functions from $X$ to $Y$.
\begin{description}
\item[$\PWD(X,Y)$]
$f$ is {\it pointwise discontinuous} if the set $C_f$ of continuity points of $f$ is dense in $X$.
\item[$\QC(X,Y)$]
$f$ is {\it quasi-continuous} if for each $x_0\in X$, 
a neighbourhood $W$ of $x_0$ and a neighbourhood $V$ of $f(x_0)$ there is a non-empty open set $W_0\subset
W$ such that $f(W_0)\subset V$.
\item[$\QC^\ast(X,Y)$]
$f$ is {\it super quasi-continuous} if $f\restr C_f$ is dense in $f$, i.e. for each $x_0\in X$, 
a neighbourhood $W$ of $x_0$ and a neighbourhood $V$ of $f(x_0)$ there is $x\in C_f\cap W$ with $f(x)\in V$.
\item[$\Cliq(X,Y)$]
Let $Y$ be a metric space. Then $f$ is {\it cliquish} if for each $x_0\in X$, $\varepsilon>0$, and a neighbourhood $W$ of $x_0$ there is a non-empty open set $W_0\subset W$ with $\osc_f(W_0)<\varepsilon$.
\end{description}
The relationships between those classes are clear. (See,
e.g., \cite{KN, TNeu}.) 
In particular, for   topological spaces $X$ and  $Y$ the
following inclusions hold (the right arrows hold under assumption that $Y$ is metric):

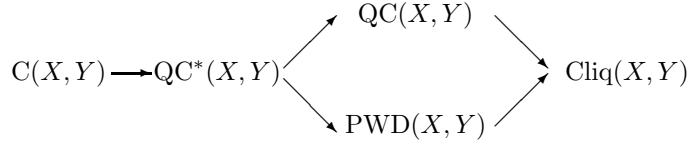
\begin{picture}(0,85)
	\put(45,40){\makebox(0,0){$\Cont(X,Y)$}}
	\put(65,40){\vector(1,0){15}}
	\put(105,40){\makebox(0,0){$\SQC(X,Y)$}}
	\put(130,42){\vector(1,1){20}}
	\put(130,38){\vector(1,-1){20}}
	\put(180,62){\makebox(0,0){$\QC(X,Y)$}}
	\put(180,20){\makebox(0,0){$\PWD(X,Y)$}}
	\put(210,62){\vector(1,-1){20}}
	\put(210,20){\vector(1,1){20}}
	\put(260,40){\makebox(0,0){$\Cliq(X,Y)$}}
\end{picture}
 Generally, all those inclusions are proper. 
 \begin{example}\label{ex1}
 	Let $\bQ=\{ q_n\colon n\in\bN\}$ be a subspace of rationals of the real line. Then the function $f\colon \bQ\to\bR$ defined as 
 	$$f(x)=\sum_{q_n\le x}\frac{1}{2^n}$$
 	is right-hand continuous, so quasi-continuous at each $x\in \bQ$, and $C_f=\emptyset$. Thus $f\in\QC(\bQ,\bR)\setminus\PWD(\bQ,\bR)$.
 \end{example}
However, if $X$ is a Baire space and $Y$ is a metric space then the diagram above reduces to the following one.
$$\Cont(X,Y)\longrightarrow \SQC(X,Y)=\QC(X,Y)\longrightarrow \Cliq(X,Y)=\PWD(X,Y).$$
This is a consequence of the fact that $f\in\Cliq(X,Y)$ iff the set $\{ x\in X\colon \osc_f(x)>0\}$ is meager in $X$.

The concept of pointwise discontinuity  has a long tradition. It appears in {\it Topology} Kuratowski's, see \cite[ Chapter I.13.III]{KKur}. The same property was rediscovered by McCoy and Hammer in \cite{MCH}, who named it {\it dense continuity}, see also \cite{HH}. The notion of quasi-continuity was introduced by Kempisty in \cite{SK}, however this idea  was perphaps used the first time by Baire in \cite{Baire} in the study of separately continuous functions. The name of super quasi-continuity was introduced by Katafiasz and Natkaniec in \cite{KN}, and this is probably the only place where it was used. The same property has been introduced and study under the name $A(X,Y)$ by Hol\'{a} and Hol\'{y} \cite{HH}. 

Notice that for any spaces $X,Y$ the equality 
$$\SQC(X,Y)=\QC(X,Y)\cap \PWD(X,Y)$$
holds. It is a simple consequence of definitions, cf. \cite{HH}.

Finally, observe that $\QC(X,Y)\not\subset\PWD(X,Y)$, even if $X$ is a Baire space. Indeed, let $X$ be the real line with the Euclidean topology, $Y$ be the Sorgenfrey line, and $f\colon X\to Y$ be the indentity. Then $f\in\QC(X,Y)\setminus \PWD(X,Y)$, see \cite{ZP}, or \cite[Example 3.1]{HH}. The question for which spaces $Y$ the inclusion $\QC(X,Y)\subset\PWD(X,Y)$ holds under assumption that $X$ is a Baire space has been intensively studied, see e.g, \cite[Section 2.3]{HHM}.

\subsection{Bors\'{i}k's properties}
Let $X$ be a topological space. The following properties of $X$ were considered by Jan Bors\'{i}k in \cite{JB}.
\begin{definition}\label{bp1}
$X$ has the property $BP_1$ whenever for every non-empty closed nowhere dense set $F\subset X$ and an open set $G$ such that $F\subset\cl(G)$ there exists a sequence $\{ K_{m,n}\colon m,n\in\bN\}$ of non-empty open subsets of $X$ such that:
	\begin{enumerate}
		\item[(i)]
		$\cl(K_{m,n})\subset G\setminus F$ for all $n,m\in \bN$;
		\item[(ii)]
		$\cl(K_{m,n})\cap\cl(K_{i,j})=\emptyset$ whenever $(m,n)\ne (i,j)$;
		\item[(iii)]
		for each $x\in X\setminus F$ there is a neighbourhood $V$ of $x$ such that the set $\{ (m,n)\colon V\cap K_{m,n}\ne\emptyset\}$ has at most one element;
		\item[(iv)]
		for each $x\in F$ and a neighbourhood $V$ of $x$, and for all $m\in\bN$ there is $n_0\in\bN$ with $K_{m,n}\subset V$ for all $n\ge n_0$. 
	\end{enumerate}
\end{definition}

\begin{definition}\label{bp2}
	$X$ has the property $BP_2$ whenever for every non-empty closed nowhere dense set $F\subset X$ and an open set $G$ such that $F\subset\cl(G)$ there exists a sequence $\{ K_{m,n}\colon m,n\in\bN\}$ of non-empty open subsets of $X$ which satisfy the conditions (i)--(iii) from Definition~\ref{bp1} and moreover,
	\begin{enumerate}
		\item[(v)]
		for each $x\in F$ and a neighbourhood $V$ of $x$, and for all $m\in\bN$ there is $n\in\bN$ with $K_{m,n}\cap V\neq\emptyset$. 
	\end{enumerate}
\end{definition}

\begin{definition}\label{bp3}
	$X$ has the property $BP_3$ whenever it satisfies all conditions of $BP_2$ with respect to any non-empty closed nowhere dense set $F$ and $G=X$.
\end{definition}

Clearly, we have implications:
$$BP_1\Longrightarrow BP_2\Longrightarrow BP_3.$$

\begin{theorem}\mbox{\cite{JB}}
\begin{enumerate}
	\item 
	Every metrizable separable space has the property $BP_1$.
	\item
	Every pseudometrizable space has the property $BP_2$.
	\item
	Every perfectly normal locally connected space has the property $BP_3$.
	\item
	The hedgehog with uncountably many spines is an example of metrizable space without property $BP_1$.
	\item
	Let $X=\beta\bN$ be the  \v{C}ech-Stone compactification of $\bN$. Then $X$ is perfectly normal space without property $BP_3$.
\end{enumerate}
\end{theorem}

Another version  of the property $BP_3$ was introduced at the same time by Zbigniew Grande \cite{ZG}.
\begin{definition}\label{bp3'}
	A topological space $X$ has the property $BP_3'$ if for every 
	nowhere dense set $F$ there is a sequence $(U_m)_{m\in\bN}$ of pairwise disjoint non-empty open sets such that
	\begin{enumerate}
		\item[(i')]
		$U_m\cap  F=\emptyset$ for all $m\in\bN$;
		\item[(ii')]
		$F\subset \cl(U_m)$ for each $m\in\bN$;
		\item[(iii')]
		$X\setminus (F\cup\bigcup_{m>0}U_m)\subset\cl(U_0)$.
	\end{enumerate}
\end{definition}

\begin{proposition}\label{p1}
	$$BP_3\Longrightarrow BP_3'.$$
\end{proposition}
\begin{proof}
Let $F_0$ be a nowhere dense set in $X$and let $F$ be a non-empty closed and nowhere dense set  such that $F_0\subset F$.  Let $\{ K_{m,n}\colon m,n\in\bN\}$ be a family witnessing the property $BP_3$ for $F$. Then $U_m:=\bigcup_{n\in\bN}K_{m,n}$ for $m>0$ and $U_0:=\int(X\setminus (F\cup\bigcup_{m>0}U_m))$ satisfy the conditions (i')--(iii').
\end{proof}

\begin{example}
	$\bR^k$ with the density topology has the property $BP_3'$, see \cite[Lemme]{ZG} if $k=1$, and \cite[Lemma 2]{GNS} for any $k\in\bN$.
\end{example}

\begin{problem}
	Does $\bR^k$ with the density topology have the property $BP_3$?
\end{problem}

\begin{definition}\mbox{(\cite{BHH}, see also \cite{LH}.)}
A topological space $X$ has  property $CP$ ($QP$) if for every nonempty
nowhere dense closed set $F\subset X$ there is a continuous (quasi-continuous)  function $g\colon X\setminus F\to [0,1]$ such that $\osc_g(x)=1$ for every $x\in F$, where $\osc_g(x)$
is the oscilation of $g$ at $x$.
\end{definition}

\begin{corollary}
$$BP_3' \Longrightarrow QP.$$
\end{corollary}
\begin{proof}
Let $F$ be a non-empty closed and nowhere dense set in $X$ and let $\{ U_m\colon m\in\bN\}$ be a family witnessing the property $BP_3'$. Let $g\colon X\setminus F\to [0,1]$ be the characteristic function of the set $U_1$, then $g$ is as we need (cf \cite[Theorem 7]{BHH}).
\end{proof}

Bors\'{i}k uses the property $BP_2$ to show that if $X$ is a pseudo-metrizable topological space then every cliquish function $f\colon X\to\bR$ can be represented as the sum of three quasi-continuous functions \cite[Theorem 1]{JB}. Following this result, an alternative proof of Theorem 3.4 from a recent paper \cite{LH} can easily be obtained.

\begin{proposition}\mbox{\cite{LH}}\label{p2}
Let $X$  be a separable metrizable space in which there is a
closed nowhere dense set $F\subset X$ with the cardinality $\co$. Then 
the family quasi-continuos non Borel measurable functions from $X$ to $[0,1]$ has cardinality
$2^\co$.
\end{proposition}
\begin{proof}
Let $\kappa$ be the size of the family of non-Borel $\QC$ functions from $X$ into $[0,1]$, and let $\lambda$ be the size of the family of all $\QC$ functions from $X$ to $[0,1]$.
Suppose that $\kappa<2^{\co}$. Since the family of all Borel functions from $X$ to $[0,1]$ has size $\le\co$, so $\lambda\le \co+\kappa<2^\co$. By Bors\'{i}k Theorem \cite[Theorem 1]{JB}, the family of all cliquish functions has size $\le \lambda^3<2^\co$. On the other hand, for every $E\subset F$, the characteristic function $\chi_E$ is cliquish, hence the family of all cliquish functions has cardinality at least  $2^\co$, a contradiction. On the other hand, if $X$ satisfies the assumptions of Proposition \ref{p2} then $X$ has cardinality $\le\co$, hence the family of all functions from $X$ to $[0,1]$ is of size $\le 2^\co$. Therefore the family of all quasi-continuous non-Borel functions from $X$ to $[0,1]$ has cardinality equal to $2^\co$.
\end{proof}

\subsection{Ideas of lineability}
During the last two decades many mathematicians have been looking at the largeness of some sets  by constructing inside them large algebraic structures, like groups, linear spaces, algebras etc.
 This approach is called {\it lineability} or {\it algebrability}. This topic was added recently to the AMS subject classification under 15A03 and 46B87. A comprehensive
 description of this concept as well as numerous examples and some general techniques can be found in the monograph \cite{ABPS} or survey article \cite{BPS}.
 
 The following terminology has been used in several results of lineability, see \cite{ABPS,BPS}. Let $\kappa$ be a cardinal number.
 \begin{itemize}
 	\item Let $V$ be a vector space over the field $\bK$ and $A\subset V$. We say that $A$ is {\it $\kappa$-lineable} if $A\cup\{ 0\}$ contains a $\kappa$-dimensional subspace of $V$.
 	\item Let $V$ be a vector space over $\bR$ and $A\subset V$. We say that $A$ is {\it positively (respectively, negatively) $\kappa$-coneable} if there is $B\subset A$ with $\card(B)=\kappa$ of linearly independent vectors in $V$ such that the positive (respectively, negative) cone generated by $B$ is contained in $A$, i.e., 
 	$$\bigcup_{n\in\bN}\left\{ \sum_{i=1}^na_ix_i\colon a_1,\ldots, a_n>0 \;\;\;(\mathrm{resp. }\;\; <0)\;\;\; \mathrm{ and }\;\;\; x_1\ldots, x_n\in B \right\}\subset A.$$
 	$A$ is {\it $\kappa$-coneable} if it is positively and negatively $\kappa$-coneable.
 	\item
 	Let $\mathfrak{A}$ be an algebra over the field $\bK\in\{ \bR,\bC\}$ and $A\subset \mathfrak{A}$.  We say that $A$ is {\it $\kappa$-algebrable} if $A\cup\{ 0\}$ contains
 	a $\kappa$-generated  subalgebra $\mathfrak{B}$ of $\mathfrak{A}$.
 	\item
 	Let $\mathfrak{A}$ be an algebra over the field $\bK\in\{ \bR,\bC\}$ and $A\subset \mathfrak{A}$.  We say that $A$ is {\it strongly $\kappa$-algebrable} if $A\cup\{ 0\}$ contains
 	a $\kappa$-generated  free subalgebra $\mathfrak{B}$.
 \end{itemize}
Recall that a set $X$ contained in  an algebra $\mathfrak{A}$ generates a free subalgebra $\mathfrak{B}$ whenever:
\begin{enumerate}
	\item the algebra $\mathfrak{B}$ is generated by $X$;
	\item for any polynomial $P$ of degree $n$ and without a constant term, and for any pairwise different $x_1,\ldots, x_n\in X$ we have $P(x_1,\ldots,x_n)=0 \Leftrightarrow P=0$. 
\end{enumerate}
 If $X$ satisfies conditions (1),(2) then we say that vectors in $X$ are {\it algebraically independent} and that $X$ is a {\it set of free generators} of $\mathfrak{B}$.

\section{Results}
\begin{proposition}\label{p3}
Let $X$ be a topological space with the property $BP_3'$ and let $F_0,F_1$ be any disjoint nowhere dense subsets of $X$. Then for $Y=\{ 0,1\}$ there is a super quasi-continuous function $g\colon X\to Y$ such that $g(x)=i$ for $x\in F_i$, $i=0,1$. 
\end{proposition}
\begin{proof}
Let $F=\cl(F_0\cup F_1)$ and let  let $\{ U_m\colon m\in\bN\}$ be a family witnessing the property $BP_3'$ for $F$. Define $g\colon X\to Y$ by the formula
$$
g(x)=\left\{\begin{array}{cl}
	1 & \text{for } x\in F_1\cup (\cl(U_1)\setminus F);\\
	0 & \text{for other } x.
\end{array}\right.
$$
Clearly $g$ satisfies the assertion of the proposition.
\end{proof} 

\begin{theorem}\label{t1}
Let $X$ be a second countable Hausdorff topological space with the property $BP_3'$ in which there is a
closed nowhere dense set $F\subset X$ with the cardinality $\co$. Then the family of all
$\SQC$ non Borel measurable functions from $X$ to $\{ 0,1\}$ has cardinality $2^\co$.	
\end{theorem}
\begin{proof}
Since $X$ is Hausdorff second countable, it has the cardinality $\le\co$ and the family of all Borel sets in $X$ has the cardinality $\le\co$. Since $\card(X)\ge\card(F)=\co$, we have $\card(X)=\co$.  Hence the family of all functions from $X$ to $\{ 0,1\}$ has cardinality $2^\co$, so the family
of $\SQC$ non Borel measurable functions from $X$ to $\{ 0,1\}$ has cardinality $\le 2^\co$. On the other hand, the family of all non-Borel subsets of $F$ has the cardinality $2^\co$ and,  by Proposition \ref{p3}, for every non-Borel subset $F_0\subset F$, the characteristic function of $F_0$ resticted to $F$ can be extended to a super quasi-continuous function from $X$ to $\{ 0,1\}$. 
Hence the family of all
$\SQC$ non Borel measurable functions from $X$ to $\{ 0,1\}$ has cardinality $\ge 2^\co$.
\end{proof}

 In proofs of next theorems we will  use the so-called method of {\it independent Bernstein sets}, see e.g. \cite{BBG}. For a non-empty set $X$ and $D\subset X$ denote $D^1=D$ and $D^0=X\setminus D$. We say that a family $\mathcal{D}$ of subsets of $X$ is {\it independent} if 
for every $\delta\colon \{ 1,\ldots,n\}\to\{ 0,1\}$ and any pairwise distinct sets $D_1,\ldots,D_n\in\mathcal{D}$ we have $D_1^{\delta(1)}\cap\ldots\cap D_n^{\delta(n)}\neq\emptyset$. Recall that for every infinite cardinal $\kappa$ there is an independent family  of $2^\kappa$ subsets of $\kappa$ \cite{FK}, see also \cite[Lemma 7.7]{Jech}.

\begin{theorem}\label{t2}
Let $X$ be a topological space with the property $BP_3'$ in which there is a 
closed nowhere dense set $F\subset X$ which is homeomorphic with the Cantor set. Then 

\begin{enumerate}
	\item There exists an additive group of size $\co$ contained in the family of non-Borel $\SQC(X,\bZ)$ functions plus the zero function.
	\item Let $\mathbb{K}\in\{ \bQ,\bR,\bC\}$. There is a $\co$-dimensional linear space (over $\mathbb{K}$) contained in the family of non-Borel $\SQC(X,\bK)$  plus the zero function.
	\item
	There exists a cone contained in the family of non-Borel $\SQC(X,\bR)$  plus the zero function that is spaned by $2^\co$ linearly independent generators. 
	\item There is a lattice of size $2^\co$ contained in the family of non-Borel $\SQC(X,\{ 0,1\})$ functions.
\end{enumerate}
\end{theorem}
\begin{proof}
Divide $F$ onto $\co$-many pairwise disjoint Bernstein sets in $F$,  $F_\alpha$, $\alpha<\co$. Let $\{ U_n\colon n\ge 0\}$ be a family witnessing the property $BP_3'$ for $F$.
Next, let 
\begin{itemize}
	\item 
$\{  N_\alpha\colon \alpha <\co\}$ be an independent family of infinite subsets of $\bN\setminus\{ 0\}$; 
\item
 $\{ \Delta_\alpha\colon \alpha<2^\co\}$ be an independent family of subsets of $\co$, and 
 \item
 for each $\alpha<2^\co$ let $D_\alpha:=\bigcup_{\xi\in\Delta_\alpha} F_\xi$.
 \end{itemize}
Notice that for all $\alpha,\beta<2^\co$ the sets $D_\alpha$, $D_\alpha\cup D_\beta$ and $D_\alpha\cap D_\beta$  are Bernstein sets in $F$. For each $\alpha<\co$ 
let $f_\alpha\colon X\to\bZ$ be the characteristic function of the set $D_\alpha\cup\bigcup\{U_n \colon n\in N_\alpha\}$:
$$
f_\alpha(x)=\left\{\begin{array}{cl}
	1 & \text{for } x\in D_\alpha;\\
	1 & \text{for } x\in U_n, n\in N_{\alpha};\\
	0 & \text{for other } x.
\end{array}\right.
$$
Then the restriction $f_\alpha\restr F$ is equal to the characteristic function of $D_\alpha$, hence it is non-Borel and therefore $f_\alpha$ is non-Borel. We will verify that $f_\alpha$ is super quasi-continuous at each $x\in X$. 
\begin{itemize}
	\item 
Since $D_\alpha\subset\cl(U_n)$ and $f_\alpha$ is constant  on $D_\alpha\cup U_n$  for any $n\in N_\alpha$, $f_\alpha$ is super quasi-continuous at each $x\in D_\alpha$. 
\item
Similarly,  $F\setminus D_\alpha\subset \cl(U_0)$ and $f_\alpha(t)=0$ for  $t\in (F\setminus D_\alpha)\cup U_0$, hence $f_\alpha$ is super quasi-continuous at each $x\in F\setminus D_\alpha$. 
\item
If $x\not\in F$ then either $x\in U_n$ for some $n\in\bN$ and then $f_\alpha$ is continuous at $x$, or $x\in X\setminus\bigcup_{n=1}^\infty U_n$ and then $x\in\cl(U_0)$ and $f_\alpha$ is constant on $\{ x\}\cup U_0$, hence it
is super quasi-continuous at $x$.
\end{itemize}

(1) Let $\mathcal{G}$ be the additive group generated by the family $\{ f_\alpha\colon \alpha<\co\}$. Clearly $\mathcal{G}$ is of size $\co$. We will show that every non-zero function $f\in\mathcal{G}$ is non-Borel and  super quasi-continuous. Fix $f\in\mathcal{G}\setminus\{ 0\}$. Then $f=\sum_{i=1}^ma_i f_{\alpha_i}$, where $m\in\bN$, $\alpha_i<\co$, and $a_i\in\bZ\setminus\{ 0\}$. Observe that $f\restr F$ is non-Borel. In fact, there are $\xi_0\in \Delta_{\alpha_1}\setminus\bigcup_{i=2}^m\Delta_{\alpha_i}$ and $\xi_1\in\co\setminus\bigcup_{i=1}^m\Delta_{\alpha_i}$. Then $f\restr F_{\xi_0}=a_1\neq 0$ and $f\restr F_{\xi_1}=0$.
Since $F_{\xi_0}$ and $F_{\xi_1}$ are disjoint Bernstein sets in $F$, so $f\restr F$ is non-Borel and consequently, $f$ is non-Borel.

We will verify that $f\in\SQC(X,\bZ)$. 
Fix $x_0\in X$ and consider a few cases. 

\begin{itemize}
	\item 
If $x_0\in D_{\alpha_i}$ for some $i\le m$ then  choose $\xi\in \Delta_{\alpha_i}\setminus \bigcup\{ \Delta_{\alpha_j}\colon 1\le j\le m, j\ne i\}$ and observe that for $n\in N_\xi$ we have $x_0\in\cl(U_n)$ and $f(x_0)=a_i=f(t)$ for $t\in U_n$, so $f$ is super  quasi-continuous at $x_0$. 
\item
If $x_0\in F\setminus \bigcup_{i\le m}D_{\alpha_i}$ then 
observe that $x_0\in\cl(U_0)$ and $f(x_0)=0=f(t)$ for $t\in U_0$, so $f$ is super quasi-continuous at $x_0$. 
\item
If $x_0\in U_n$ for some $n$, then since $f$ is constant on $U_n$, $f$ is super quasi-continuous at $x_0$. 
\item
Finally, assume that $x_0\not\in F\cup\bigcup_{n\in\bN}U_n$. Then 
$x_0\in\cl(U_0)$ and $f(x_0)=0=f(t)$ for $t\in U_0\subset C(f)$, hence $f$ is super quasi-continuous at $x_0$.
\end{itemize}
In conclusion, $f$ is super quasi-continuous at each point $x\in X$.

(2) Let $\mathbb{K}\in\{ \bQ,\bR,\bC\}$ and let $V$ be a linear space over $\mathbb{K}$ generated by the family $\{ f_\alpha\colon \alpha<\co\}$. 
 We will verify that $f_\alpha$'s are linearly independent. Assume that $\sum_{i=1}^ma_i f_{\alpha_i}=0$, where $a_i\in \mathbb{K}$ and $\alpha_i<\co$ are pairwise different. For every $i\le m$ choose $\xi\in \Delta_{\alpha_i}\setminus\bigcup_{j\ne i}\Delta_{\alpha_i}$. Then for each  $x\in F_{\xi}$ we have $\sum_{i=1}^ma_i f_{\alpha_i}(x)=a_i$, hence $a_i=0$. Finally let $f$ be a non-zero linear combination of $f_\alpha$'s. Then $f=\sum_{i=1}^n a_i f_{\alpha_i}$, where $a_i\in \mathbb{K}\setminus\{ 0\}$ and $\alpha_i<\co$ are pairwise different.
One can verify, as in the point (1), that $f$ is non-Borel and $\SQC$.

(3) 
For each $0<\alpha<2^\co$ let $g_\alpha\colon X\to \{0,1\}$ be defined as
$$
g_\alpha(x)=\left\{\begin{array}{cl}
	1 & \text{for } x\in D_\alpha;\\
	1 & \text{for } x\in U_n, n\in N_{0};\\
	0 & \text{for other } x.
\end{array}\right.
$$
First, observe that $g_\alpha$'s are non-Borel. In fact, $g_\alpha(x)=1$ for $x\in D_\alpha$ and $g_\alpha(x)=0$ for $x\in F\setminus D_\alpha$. Since $D_\alpha$ is non-Borel, $g_\alpha\restr F$ is non-Borel and consequently, $g_\alpha$ is non-Borel, too. Next, we will see that $g_\alpha\in\SQC$. Fix $x\in X$ and consider a few cases. 
\begin{itemize}
	\item 
 Let $x\in D_\alpha$. Fix any $n\in N_{0}$ and observe that $x\in\cl(U_n)$ and $g_\alpha(x)=1=g(t)$ for $t\in U_n$.  Since $U_n$ is open, $g_\alpha$ is super quasi-continuous at $x$.
\item
Let $x\in U_n$ for some $n$. Since $U_n$ is open and $g_\alpha$ is constant on $U_n$, so $g_\alpha$ is $\SQC$ at $x$.
\item
If $x\in (F\setminus D_\alpha)\cup [X\setminus (F\cup\bigcup_{n\ge 0}U_n)]$, then $x\in\cl(U_0)$ and $g_\alpha(x)=0=g_\alpha(t)$ for $t\in U_0$, hence $g_\alpha$ is super quasi-continuous at $x$.
\end{itemize}

Now, let $g=\sum_{i=1}^n a_ig_{\alpha_i}$,  where $a_i\in \mathbb{R}$, $a_i>0$,  and $0<\alpha_1,\ldots,\alpha_n<2^\co$ are pairwise different.
Let $\xi\in \bigcap_{i=1}^n\Delta_{\alpha_i}$ and $\eta\in\co\setminus\bigcup_{i=1}^n\Delta_{\alpha_i}$, then $g(x)=\sum_{i=1}^na_i>0$ for $x\in F_\xi$ and $g(x)=0$ for $x\in F_\eta$, hence $g\restr F$ is non-Borel, and consequently $g$ is non-Borel, too. To see that $g\in\SQC$  fix $x\in X$ and consider a few cases. 

\begin{itemize}
	\item 
First, assume that $x\in F$. Let $I_x=\{ i\le n\colon x\in D_{\alpha_i}\}$. We will conside two subcases.
\begin{itemize}
	\item 
If $I_x\ne\emptyset$ then, since the family $\{ N_\alpha\colon\alpha<\co\}$ is independent, we can  choose  $n\in N_0\cap \bigcap_{i\in I_x}N_i\setminus\bigcup_{i\le n, i\not\in I_x}N_i$. 
Then 
$$
g_{\alpha_i}(x)=\left\{\begin{array}{cl}
	1 & \text{for } i\in I_x;\\
	0 & \text{for } i\in \{i\le n, i\not\in I_x\},
\end{array}\right.
$$
and $g_{\alpha_i}(t)=g_{\alpha_i}(x)$ for $t\in U_n$ and $i\le n$.
Hence $g(x)=\sum_{i=1}^n a_ig_{\alpha_i}(x)=\sum_{i\in I_x}a_i=\sum_{i\le n}a_ig_{\alpha_i}(t)=g(t)$ for $t\in U_n$, so
 $g$ is constant on the set $\{ x\}\cup U_n$. Thus $g$ is $\SQC$ at $x$.
 \item
If $x\in F$ and $I_x=\emptyset$ then $g(x)=0$. Since $g(t)=0$ for $t\in U_0$ and $x_0\in\cl(U_0)$, $g$ is $\SQC$ at $x$. 
\end{itemize}
\item
Finally, assume that
 $x\not\in F$.  Then either $x\in U_n$ for some $n\in\bN$ and $U_n$ is open and $g$ is constant on $U_n$, or 
$x\in\cl(U_0)$ and then $g(x)=0=g(t)$ for $t\in U_0$.  In both cases, $g$ is $\SQC$ at $x$. 
\end{itemize}
In conclusion, $g$ is super quasi-continuous at any $x\in X$.
Hence the family non-Borel $\SQC$ functions is positively $2^\co$-coneable. In the same way one can verify that it is also negatively $2^\co$-coneable.

(4)
It is easy to observe that for any $\alpha,\beta<2^\co$ we have 
$$
\max(g_\alpha,g_\beta)(x)=\left\{\begin{array}{cl}
		1 & \text{for } x\in D_\alpha\cup D_\beta;\\
		1 & \text{for } x\in U_n, n\in N_{0};\\
		0 & \text{for other } x.
	\end{array}\right.
$$
and 
$$
\min(g_\alpha,g_\beta)(x)=\left\{\begin{array}{cl}
		1 & \text{for } x\in D_\alpha\cap D_\beta;\\
		1 & \text{for } x\in U_n, n\in N_{0};\\
		0 & \text{for other } x.
	\end{array}\right.
$$
so $\min(g_\alpha,g_\beta)$ and $\max(g_\alpha,g_\beta)$ are non-Borel and $\SQC$.
\end{proof}

Now we will consider a functions between topological spaces $X$ and $Y$. The following fact follows from Proposition \ref{p3}.
\begin{corollary}\label{c1}
Let $X$ be a space with the property $BP_3'$ and let $Y$ be a Hausdorff space with $|Y|\ge 2$. Then for all disjoint nowhere dense sets $F_0,F_1\subset X$ and for any different $y_0,y_1\in Y$ there is a super quasi-continuous function $g\colon X\to Y$ such that $g(x)=y_i$ for $x\in F_i$, $i=0,1$. 	
\end{corollary} 
\begin{proof}
Let $F_0$ and $F_1$ be disjoint nowhere dense subsets of $X$ and let $\tilde{g}\colon X\to\{ 0,1\}$ be a super quasi-continuous function such that $\tilde{g}(x)=i$ for $x\in F_i$, $i=0,1$, as in Proposition \ref{p3}. Fix different $y_0,y_1\in Y$. We consider $\{ 0,1\}$ as a topological space equipped in the discrete topology. Then the map $\varphi\colon \{0,1\}\to Y$ defined by $\varphi(i)=y_i$, $i=0,1$, is a continuous embedding, and $g=\varphi\circ \tilde{g}$ is as we need.
\end{proof}

As a consequence of the Corollary \ref{c1} we obtain the following facts which are a generalization of \cite[Proposition 3.3, Theorem 3.4]{LH}.
\begin{proposition}\label{p4}
Let $X$ be a separable topological space with the property $BP_3'$ in which there is a
closed nowhere dense set $F\subset X$ with the cardinality $\co$, and let $Y$ be a Hausdorff space with $|Y|\ge 2$. Then the family of all
super quasi-continuous non Borel measurable functions from $X$ to $Y$ has cardinality $\ge 2^\co$. Moreover, if $2\le |Y|\le\co$, then the cardinality of this family is equal to $2^\co$.
\end{proposition}
\begin{proof}
Let $\F$ denote the family 	of all
super quasi-continuous non Borel measurable functions from $X$ to $Y$. The inequality $|\F|\ge 2^\co$ yields from Corollary \ref{c1}. Since $X$ is separable, so $|X|\le\omega^\omega=\co$. Since there is $F\subset X$ with $|F|=\co$, we have $X=\co$. Hence $|\F|\le |Y^X|\le \co^\co=2^\co$.
\end{proof}

\begin{corollary}
Let $X$ be an uncountable Polish space and let $Y$ be a Hausdorff space with $|Y|\ge 2$. Then the family of all
quasi-continuous non Borel measurable functions from $X$ to $Y$ has cardinality $\ge 2^\co$. Moreover, if $2\le |Y|\le\co$, then the cardinality of this family is equal to $2^\co$.
\end{corollary}

The proof of the next theorem is similar to the proof  of Theorem \ref{t2}. 
\begin{theorem}\label{t3}
	Let $X$ be a  topological space with the property $BP_3'$ in which there is a
	closed nowhere dense set $F\subset X$ which is homeomorphic with the Cantor set. Then 
	
	\begin{enumerate}
		\item For any topological group $(Y,+)$ there exists an additive group of size $\co$ contained in the family of non-Borel $\SQC(X,Y)$ functions plus the zero function.
		\item Let $Y$ be a linear space over $\mathbb{K}\in\{ \bQ,\bR,\bC\}$. There is a $\co$-dimensional linear space over $\mathbb{K}$ contained in the family of non-Borel $\SQC(X,Y)$ functions plus the zero function.
		\item If $Y$ is a lattice and $|Y|\ge 2$   then there is a  lattice of size $\co$ contained in the family of non-Borel $\SQC(X,Y)$ functions.
	\end{enumerate}
\end{theorem}

Recall that a Riesz space is an ordered vector space that is also a lattice.
\begin{theorem}\label{t4}
	Let $X$ be a topological space with the property $BP_3'$ in which there is a
closed nowhere dense set $F\subset X$ which is homeomorphic with the Cantor set and let $Y$ be a Riesz space over $\mathbb{R}$ with the order $\le$. Then there exists a $\co$-dimensional Riesz space over $\bR$ with the order $f\le g$ iff $f(x)\le g(x)$ for every $x\in X$, contained in the family of non-Borel $\SQC(X,Y)$ functions plus the zero function.
\end{theorem} 	
\begin{proof}
Divide $F$ onto $\co$-many pairwise disjoint Bernstein sets in $F$,  $F_\alpha$, $\alpha<\co$. Let $\{ U_n\colon n\ge 0\}$ be a family witnessing the property $BP_3'$ for $F$.
Let $\{  N_\alpha\colon \alpha <\co\}$ be an independent family of infinite subsets of $\bN\setminus\{ 0\}$. Fix $y\in Y\setminus\{ 0\}$. For every $\alpha<\co$ define a function $f_\alpha\colon X\to Y$ via
$$
f_\alpha(x)=\left\{\begin{array}{cl}
	y & \text{for } x\in F_\alpha;\\
	y & \text{for } x\in U_n, n\in N_{\alpha};\\
	0 & \text{for other } x.
\end{array}\right.
$$	
One can verify, similarly to the proof of part (2) in Theorem \ref{t2}, that the functions $f_\alpha$'s generate a linear space $L$ contained in the  family of non-Borel $\SQC(X,Y)$ functions plus the zero function. We will show that $L$ with the order $\le$ is a lattice. Fix $f,g\in L$. Then $f=\sum_{i=1}^n a_if_{\alpha_i}$, $g=\sum_{i=1}^n b_if_{\alpha_i}$ for some $\alpha_1<\ldots\alpha_n<\co$ and $(a_1,\ldots,a_n),(b_1,\ldots,b_n)\in\bR^n$, and $\max(f,g)=\sum_{i=1}^n\max(a_i,b_i)f_{\alpha_i}$, $\min(f,g)=\sum_{i=1}^n\min(a_i,b_i)f_{\alpha_i}$ belong to $L$.
\end{proof}

Note that the size $\co$ in Theorem \ref{t3} is maximal possible if we assume that $X$ is a Polish space and $Y$ is a topological group of size $\le\co$.
\begin{theorem}\label{t4}
Assume that $X$ is a space with the density $\kappa$ and $(Y,+)$ be a topological group with $\card(Y)=\lambda$. Then any group contained in the family $\SQC(X,Y)$ has cardinality $\le \lambda^\kappa$. In particular, if $X$ is a separable Baire space and $Y$ is a metrizable with $|Y|\le\co$ then any group contained in the family $\QC(X,Y)$ has cardinality $\le \co$.
\end{theorem} 
\begin{proof}
Suppose that there exists a group $\cG$ contained in the family $\SQC(X,Y)$ with $\card(\cG)>\lambda^\kappa$. Let $D$ be a dense subset of $X$ with $\card(D)=\kappa$. Since the family of all functions from $D$ to $Y$ has cardinality $\lambda^\kappa$, there are different $f, g\in\cG$ which agree on $D$. Then $f-g\in \cG$, $f-g\ne 0$ and $(f-g)\restr D=0$. Since $D$ is dense in $X$, $(f-g)(x)=0$ for every $x\in X$ at which $f-g$ is continuous. Since $f-g\in\SQC(X,Y)$, so $f-g=0$, a contradiction.

The second part of the assertion follows from the obvious fact that for functions between a Baire space and a metric space $Y$ the notions of quasi-continuity and super quasi-continuity are equivalent.
\end{proof}

As a consequence of Theorems \ref{t3} and \ref{t4} we obtain the following corollary. If $X=Y=\bR$ then this fact can be found in \cite[Theorem 4]{JW}.
\begin{corollary}
	Assume that $X$ is an uncountable Polish space and $(Y,+)$ is a metric linear space over $\mathbb{K}\in\{ \bQ,\bR,\bC\}$  with $|Y|\le \co$. Then the lineability of the family of all non-Borel $\QC(X,Y)$ functions is equal to $\co$.
\end{corollary}

J. W\'{o}dka noticed at the end of \cite{JW} that the family of all quasi-continuous functions from $\bR$ to $\bR$ that are not Lebesgue measurable plus the zero function contains a $\co$-generated free algebra. We will generalize this fact in the next proposition.

\begin{theorem}
Let $X$ be a topological space with the property $BP'_3$ in which there is a 
closed nowhere dense set $F\subset X$ which is homeomorphic with the Cantor set. Then there exists a $\co$-generated free algebra (over $\mathbb{R}$) contained in the family of non-Borel $\SQC(X,\bR)$  functions plus the zero function.
\end{theorem}
\begin{proof}
We use the method {\it of exponential like functions} proposed in \cite[Proposition~7]{BBF}. 
It suffices to show that there is a non-Borel $\SQC(X,\bR)$ function  such that $h\circ f$
is non-Borel $\SQC(X,\bR)$ for every exponential like function $h\colon\bR\to\bR$.  Then functions of the form $x\mapsto \exp(\beta g(x))$, where parameters $\beta$ are taken from a Hamel basis of $\bR$ over $\bQ$, are free generators of an algebra
included in the set of non-Borel $\SQC(X,\bR)$ functions. In \cite{BBF}, authors apply this method to real functions defined on the unit interval, but it is easy to see that it can also be used under more general assumptions on the domain. Recall that a function $h$ is exponential like whenever
$$h(x):=\sum_{i=1}^m a_i\exp(\beta_i x)$$
where $\beta_i>0$ for $i=1,\dots ,m$.

Divide $F$ onto $\co$-many Bernstein sets  $F_\alpha$, $\alpha<\co$, in $F$. Let $\{ U_n\colon n\ge 0\}$ be a family witnessing the property $BP_3'$ for $F$ and $(q_n)_{n\in\bN}$ be one-to-one sequence of rational numbers from the interval $[0,1]$ with $q_0=0$. Let $[0,1]=\{ y_\alpha\colon\alpha<\co\}$ and $f_0\colon F\to [0,1]$ be the function such that  $f_0(F_\alpha)=\{ y_\alpha\}$. Notice that $f_0$ is non-Borel. Define $f\colon X\to [0,1]$ via
$$ f(x)=\left\{\begin{array}{cl}
	f_0(x) & \text{for } x\in F;\\
	q_n & \text{for } x\in U_n, n\in \bN;\\
	0 & \text{for other } x.
\end{array}\right.
$$
Since $f\restr F=f_0$ is non-Borel, $f$ is non-Borel, too. One can easily see that $f$ is $\SQC$ at every point $x\in X\setminus F$. 
To see that $f$ is $\SQC$ at a point $x\in F$, fix an $\varepsilon>0$, a neighbourhood $V$ of $x$, and $n\in\bN$ such that $|f_0(x)-q_n|<\varepsilon$. Then $x\in\cl(U_n)$, hence $V\cap U_n$ is an open non-empty set on which $f$ is continuous.

We will verify that for every expodential like function $h\colon \bR\to\bR$ the superposition $h\circ f$ is non-Borel and super quasi-continuous. Let $h(x):=\sum_{i=1}^m a_i\exp(\beta_i x)$. Then $h$ is continuous and finite-to-one, see \cite[Lemma 8]{BBF}.
Fix $x_0\in X$, a neighbourhood $V$ of $x_0$, and $\varepsilon>0$. Since $h$ is continuous at $f(x_0)$, there is an open set $W\subset\bR$ such that $f(x_0)\in W$ and $|h(y)-h(f(x_0))|<\varepsilon$ for $y\in W$.
Since $f$ is super quasi-continuous at $x_0$, there is $x\in V$ such that $f$ is continuous at $x$ and $f(x)\in W$. Then $h\circ f$ is continuous at $x$ and $|h(f(x))-h(f(x_0))|<\varepsilon$, hence $h\circ f$ is super quasi-continuous at $x_0$. Now we will verify that $h\circ f$ is non-Borel. Let $z\in\rng(h\circ f)\restr F$. Then $h^{-1}(z)$ is non-empty and finite, hence $(h\circ f)^{-1}(z)$ is a sum of finite many Bernstein sets $F_{\alpha_1}\cup\ldots\cup F_{\alpha_n}$, so it is a Bernstein set in $F$. Thus $(h\circ f)\restr F$ is non-Borel and therefore $h\circ f$ is non-Borel, too. 
\end{proof}

Now, we will return once again to the example of the density topology $d$ on $\bR^k$, see e.g. \cite{AB}. Recall that $X=(\bR^k,d)$ is a Baire space,  $f\in \PWD(X,\bR)$ iff it is a Lebesgue measurable, and $f\in \SQC(X,\bR)$ iff $f$ is measurable and non-degenerate at every point $x\in\bR^k$, see \cite[Lemmas 1,2]{GNS}. A function $f\colon\bR^k\to\bR$ is degenerate at $x\in\bR^k$ if there is a neighbourhood $U$ of $f(x)$ such that the set $f^{-1}(U)$ has density $0$ at $x$, see \cite{ZG1}.
\begin{example}
	Let $X=(\bR^k,d)$, then any additive group contained in the family $\SQC(X,\bR)$ has cardinality $\le \co$. 
\end{example}
\begin{proof}	
	Suppose that there is a group $\cG$ contained in $\SQC(X,\bR)$ with $\card(\cG)>\co$. Since every $f\in\cG$ is measurable, so it is equal almost everywhere to some Borel function $\hat{f}:\bR^k\to\bR$ (with respect to the Euclidean topology in $\bR^k$). Since the family of Borel functions from $\bR^k$ to $\bR$  has cardinality $\co$, there are different functions $f_0,f_1\in\cG$ with $\hat{f_0}=\hat{f_1}$, and consequently, $0\ne f_0-f_1\in\cG$,  but $f_0-f_1=0$ a.e. But then $f_0-f_1$ is degenerate at each  $x\in\bR^k$ with $f_0(x)-f_1(x)\ne 0$, contrary to $f_0-f_1\in\SQC(X,\bR)$.
\end{proof}
 
Observe that the situation is different if we consider  pointwise discontinuous functions.
Assume that $\bK\in\{\bR,\bC\}$,  and let $C\subset\bR$ be a perfect set. Recall that there is a $2^\co$-generated free algebra contained in the family of all functions $f\colon C\to\bK$ that maps every perfect set $P\subset C$ onto $\bK$ (plus the zero function), see \cite[Theorem 2.2]{BGP}. Notice that every such function is non-Borel.
\begin{proposition}
	Let $\bK\in\{ \bR,\bC\}$ and $X$ be a topological space with the property $BP'_3$ in which there is a 
closed nowhere dense set $F\subset X$ which is homeomorphic with the Cantor set. Then there exists a $2^\co$-generated free algebra (over $\mathbb{K}$) contained in the family of non-Borel $\PWD(X,\bK)$  functions plus the zero function.
\end{proposition}
\begin{proof}
	Let $\F=\{ f_\alpha\colon \alpha<2^\co\}\subset \bK^F$ be an algebraically independent family of non-Borel functions. For each $\alpha<2^\co$ let $g_\alpha\colon X\to\bK$ be defined by 
$$ g_\alpha(x)=\left\{\begin{array}{ll}
	f_\alpha(x) & \text{for } x\in F;\\
		0 & \text{for } x\in X\setminus F.
	\end{array}\right.
	$$	
Since $g_\alpha\restr F=f_\alpha$ for all $\alpha<\co$, then $\{ g_\alpha\colon\alpha<2^\co\}$ has size $2^\co$, is linearly independent, and generates a free algebra $\mathfrak{A}$. Moreover, every $f\in\mathfrak{A}$ equals $0$ on an open dense subset of $X$, hence it is pointwise discontinuous.
\end{proof}	

Finally, assume that $X$ is not a Baire space and consider the family $\QC(X,\bR)\setminus\SQC(X,\bR)$.
\begin{example}
There is a $\co$-generated free algebra contained in the family $\QC(\bQ,\bR)\setminus\SQC(\bQ,\bR)$ plus the zero function. 
\end{example}
\begin{proof}
Let $f\colon\bQ\to [0,1]$ be the function defined in Example~\ref{ex1}. We will verify that $h\circ f\in \QC(X,\bR)\setminus\SQC(X,\bR)$ for every exponential like function $h\colon\bR\to\bR$. Clearly, $h\circ f$, as a composition of quasi-continuous function with a continuous one, is quasi-continuous, see e.g. \cite{TNeu}. Since $h$ is exponential-like, there is a finite sequence $0=y_0<y_1<\ldots<y_k=1$ such that $h$ is strictly monotone on each interval $[y_{i-1},y_i]$, $i\le k$, see \cite[Lemma 1.4]{BBF}. Since $\rng(f)\subset [0,1]$ is infinite, there is $j\le k$ such that the set $A:=\rng(f)\cap [y_{j-1},y_j]$ is infinite. Then $h\restr [y_{j-1},y_j]$ is a homeomorphism, the set $B:=f^{-1}([y_{j-1},y_j])$ is a non-degenerate  interval in $\bQ$, and $h\circ f$ is continuous at no point of $B$, so $h\circ f\not\in\SQC(\bQ,\bR)$.
\end{proof}

The result above is optimal in the sense that every algebra contained in the family $\QC(\bQ,\bR)$ is of size $\le\co$, because $\card(\bR^\bQ)=\co$. It would be interesting to know if there are a space $X$ and a cardinal $\kappa>\co$ for which  the family $\QC(X,\bR)\setminus\SQC(X,\bR)$ is strongly $\kappa$-algebrable.

\subsection*{Acknowledgements}
We would like to thank Professor  \u{L}ubica Hol\'{a} and the anonymous reviewer for
many useful remarks that have helped us to improve the former version of the
paper.



\section*{Statements \& Declarations}
\begin{enumerate}
	\item The author declares that no funds, grants, or other support were received during the preparation of this manuscript.
	\item The author has no relevant financial or non-financial interests to disclose.
\end{enumerate}
\section*{Data Availability Statement}
Statement Data sharing is not applicable to this article as no datasets were generated or analyzed during the preparation of the paper.

\end{document}